\documentclass{amsart}

\input xy
\xyoption{all}

\usepackage{geometry}
\usepackage{amsmath}
\usepackage{amssymb}
\usepackage{cite}
\usepackage{amsthm}  
\usepackage{tikz}
\usetikzlibrary{er,positioning}

\newtheorem{same}{This should never appear}[section]
\newtheorem{defin}[same]{Definition}

\newtheorem{remark}[same]{Remark}
\newtheorem{theorem}[same]{Theorem}
\newtheorem{example}[same]{Example}
\newtheorem{lemma}[same]{Lemma}
\newtheorem{fact}[same]{Fact}
\newtheorem{question}[same]{Question}
\newtheorem{cor}[same]{Corollary}
\newtheorem{prop}[same]{Proposition}

\newtheorem{hypothesis}[same]{Hypothesis}
\newtheorem{nota}[same]{Notation}

\newtheorem{defin*}{Definition}
\newtheorem*{theorem*}{Theorem}

\makeatletter
\newcommand{\skipitems}[1]{%
  \addtocounter{\@enumctr}{#1}%
}
\newcommand{\bb}{\mathbf{b}}
\newcommand{\rest}{\mathord{\upharpoonright}}

\newcommand{\K}{\mathbf{K}}

\newcommand{\LS}{\operatorname{LS}}
\newcommand{\rc}{(|R| + \aleph_0)^+}
\newcommand{\tilk}{\tilde{T}_\K}

\newcommand{\leap}[1]{\le_{#1}}

\newcommand{\lea}{\leap{\K}}

\newcommand{\gtp}{\mathbf{gtp}}
\newcommand{\gS}{\mathbf{gS}}

\newcommand{\pullbackcorner}[1][dl]{\save*!/#1-3pc/#1:(1,-1)@^{|-}\restore} 

\DeclareMathOperator{\pp}{pp}    
\DeclareMathOperator{\cof}{cf}    

\newbox\noforkbox \newdimen\forklinewidth
\setbox0\hbox{$\textstyle\smile$}
\setbox1\hbox to \wd0{\hfil\vrule width \forklinewidth depth-2pt
 height 10pt \hfil}
\setbox\noforkbox\hbox{\lower 2pt\box1\lower 2pt\box0\relax}
\def\unionstick{\mathop{\copy\noforkbox}\limits}

\def\nonfork_#1{\unionstick_{\textstyle #1}}

\setbox0\hbox{$\textstyle\smile$}
\setbox1\hbox to \wd0{\hfil{\sl /\/}\hfil}
\setbox2\hbox to \wd0{\hfil\vrule height 10pt depth -2pt width
              \forklinewidth\hfil}
\newbox\doesforkbox
\setbox\doesforkbox\hbox{\lower 2pt\box1 \lower 2pt\box2\lower2pt\box0\relax}
\def\nunionstick{\mathop{\copy\doesforkbox}\limits}

\def\fork_#1{\nunionstick_{\textstyle #1}}

\newcommand{\s}{\mathfrak{s}}
\newcommand{\dnf}{\unionstick}

\title{Some stable non-elementary classes of modules}

\author{Marcos Mazari-Armida}
\email{mmazaria@andrew.cmu.edu}
\urladdr{http://www.math.cmu.edu/~mmazaria/ }
\address{Department of Mathematical Sciences \\ Carnegie Mellon
University \\ Pittsburgh, Pennsylvania, USA}

\begin{document}

\begin{abstract}

Fisher \cite{fisher} and Baur \cite{baur}  showed independently in the seventies that if $T$ is a complete first-order theory extending the theory of modules, then the class of models of $T$ with pure embeddings is stable.  In \cite[2.12]{maztor}, it is asked if the same is true for any abstract elementary class $(K, \leq_p)$ such that $K$ is a class of modules and $\leq_p$ is the pure submodule relation. In this paper we give some instances where this is true:

\begin{theorem} Assume $R$ is an associative ring with unity.
Let $(K, \leq_p)$ be an AEC  such that $K \subseteq R\text{-Mod}$ and $K$ is closed under finite direct sums, then:
\begin{itemize}
 \item If $K$ is closed under pure-injective envelopes, then $\K$ is $\lambda$-stable for every $\lambda \geq \LS(\K)$ such that $\lambda^{|R| + \aleph_0}= \lambda$.
 \item If $K$ is closed under pure submodules and pure epimorphic images,  then $\K$ is $\lambda$-stable for every $\lambda$ such that $\lambda^{|R| + \aleph_0}= \lambda$.

\item Assume $R$ is Von Neumann regular. If $\K$ is closed under submodules and has arbitrarily large models, then $\K$ is $\lambda$-stable for every $\lambda$ such that $\lambda^{|R| + \aleph_0}= \lambda$.
\end{itemize}
\end{theorem}
 
As an application of these results we give new characterizations of noetherian rings, pure-semisimple rings, Dedekind domains, and fields via superstability. Moreover, we show how these results can be used to show a link between being \emph{good} in the stability hierarchy and being \emph{good} in the axiomatizability hierarchy.

 Another application is the existence of universal models with respect to pure embeddings in several classes of modules. Among them,  the class of flat modules and the class of $\s$-torsion modules.

.

\end{abstract}


\maketitle

{\let\thefootnote\relax\footnote{{AMS 2020 Subject Classification:
Primary: 03C48 Secondary: 03C45, 03C60, 13L05, 16D10, 16P40.
}}}

\tableofcontents

\section{Introduction}

An abstract elementary class $\K$ (AEC for short) is a pair $\K=(K \lea)$ where $K$ is a class of structures and $\lea$ is a partial order on $K$ extending the substructure relation such that $\K$ is closed under direct limits and satisfies the coherence property and an instance of the Downward L\"{o}wenheim-Skolem theorem. These were introduced by Shelah in \cite{sh88}. In this paper, we will study  AECs of modules with respect to pure embeddings, i.e., classes of the form $(K, \leq_p)$ where $K$ is a class of $R$-modules for a fixed ring $R$ and $\leq_p$ is the pure submodule relation.

Fisher \cite{fisher} and Baur \cite[Theo 1]{baur} showed independently in the seventies that if $T$ is a complete first-order theory extending the theory of modules, then $(Mod(T), \leq_p)$ is $\lambda$-stable for every $\lambda$ such that $\lambda^{|R| + \aleph_0}= \lambda$. A modern proof can be consulted in \cite[3.1]{prest}. After realizing that many other classes of modules with pure embeddings were stable such as: abelian groups \cite[3.16]{kuma}, torsion-free abelian groups \cite[0.3]{baldwine}, torsion abelian groups \cite[4.8]{maztor}, complete elementary classes of torsion modules \cite[5.16]{bontor}, reduced torsion-free abelian groups \cite[1.2]{sh820}, definable subclasses of modules \cite[3.16]{kuma}, and flat $R$-modules \cite[4.3]{lrvcell}; it was asked in \cite[2.12]{maztor} the following question:

\begin{question}\label{mainq} Let $R$ be an associative ring with unity.
If $(K, \leq_p)$ is an abstract elementary class such that $K \subseteq R\text{-Mod}$, is $(K, \leq_p)$ stable? Is this true if $R=\mathbb{Z}$? Under what conditions on $R$ is this true?
\end{question}

In this paper, we show that many classes of modules are stable. The way we approach the problem is by showing that if the class has some nice algebraic properties then it has to be stable. This approach is new, covers most of the examples known to be stable\footnote{The only set of examples that this approach does not cover is that of classes axiomatizable by complete first-order theories (both classical and their torsion part as in \cite[\S 5]{bontor}).} and can be used to give many new examples. Prior results focused on syntactic characterizations of the classes or only obtained stability results for specific examples.

Firstly, we study classes closed under direct sums and pure-injective envelopes. These include absolutely pure modules, locally injective modules, locally pure-injective modules, reduced torsion-free groups, and definable subclasses of modules (see Example \ref{ex1}).

\textbf{Theorem \ref{stat1}.} \textit{Assume $\K= (K , \leq_{p})$ is an AEC with $K \subseteq R\text{-Mod}$ for $R$ an associative ring with unity such that $K$ is closed under direct sums and pure-injective envelopes. If $\lambda^{|R| + \aleph_0}=\lambda$ and $\lambda \geq \LS(\K)$, then $\K$ is $\lambda$-stable.}

By characterizing the limit models in these classes (Lemma \ref{bigpi} and Lemma \ref{countablelim}), we are able to obtain new characterizations of noetherian rings, pure-semisimple  rings, Dedekind domains, and fields via superstability. An example of such a result is the next assertion which extends \cite[4.30]{maz1}. 

\textbf{ Theorem \ref{absss}.}\textit{
Let $R$ be an associative ring with unity. $R$ is left noetherian if and only if the class of absolutely pure left $R$-modules with pure embeddings is superstable.}

Moreover, the above result can be used to show a link between being \emph{good} in the stability hierarchy and being \emph{good} in the axiomatizability hierarchy. More precisely, if the class of absolutely pure modules with pure embeddings is superstable, then it is first-order axiomatizable (see Corollary \ref{good}).

The results for these classes of modules can also be used to partially solve Question \ref{mainq} if one substitutes \emph{stable} for \emph{superstable}.

\textbf{ Lemma \ref{pss}.}\textit{
Let $R$ be an associative ring with unity. The following are equivalent. 
\begin{enumerate}
\item $R$ is left pure-semisimple.
\item Every AEC $\K= (K , \leq_{p})$ with $K \subseteq R\text{-Mod}$, such that $K$ is closed under direct sums, is superstable.
\end{enumerate}
}

Secondly, we study classes closed under direct sums, pure submodules, and pure epimorphic images. These include flat modules, torsion abelian groups, $\s$-torsion modules, and any class axiomatized by an $F$-sentence (see Example \ref{ex2}).

\textbf{Theorem \ref{statf}.} \textit{Assume $\K= (K , \leq_{p})$ is an AEC with $K \subseteq R\text{-Mod}$ for $R$ an associative ring with unity such that $K$ is closed under direct sums, pure submodules, and pure epimorphic images. If $\lambda^{|R| + \aleph_0}=\lambda$, then $\K$ is $\lambda$-stable.}

This result can be used to construct universal models with respect to pure embeddings. In particular, we obtain the next result which extends \cite[1.2]{sh820}, \cite[4.6]{mazf}, and \cite[3.7]{maztor}. 

\textbf{Corollary \ref{universal2}.}\textit{ Let $R$ be an associative ring with unity. 
If $\lambda^{|R| + \aleph_0}=\lambda$ or $\forall \mu < \lambda( \mu^{|R| + \aleph_0} <
\lambda)$, then there is a universal model in the class of flat $R$-modules with pure embeddings and in the class of $\s$-torsion $R$-modules with pure embeddings of cardinality $\lambda$.}

Finally, we study classes of modules that are closed under pure submodules and that are contained in a well-understood class of modules which is closed under pure submodules and that admits intersections. The main examples for this case are subclasses of the class of torsion-free groups such as $\aleph_1$-free-groups and finitely Butler groups (see Example \ref{ex3}). 

We use the results obtained for these classes of modules to provide a partial solution to Question \ref{mainq}. 

\textbf{ Lemma \ref{von}.}\textit{
Assume $R$ is a Von Neumann regular ring.
If $K$ is closed under submodules and has arbitrarily large models, then $\K =(K, \leq_p)$ is $\lambda$-stable if $\lambda^{ |R| + \aleph_0} = \lambda$.}

The paper is organized as follows. Section 2 presents necessary background. Section 3 studies classes closed under direct sums and pure-injective envelopes. Section 4 studies classes closed under direct sums, pure submodules, and pure epimorphic images. Section 5 studies classes of modules that are closed under pure submodules and that are contained in a well-understood class of modules which is closed under pure submodules and that admits intersection.

This paper was written while the author was working on a Ph.D. under the direction of Rami Grossberg at Carnegie Mellon University and I would like to thank Professor Grossberg for his guidance and assistance in my research in general and in this work in particular. I would like to thank Thomas G. Kucera for letting me include Lemma \ref{stru1} in this paper. I would like to thank John T. Baldwin, Ivo Herzog, Samson Leung, and Philip Rothmaler  for comments that help improve the paper. I  am  grateful  to  an anonymous referee  for  many  comments  that  help improve the paper.

\section{Preliminaries}

In this section, we recall the necessary notions from abstract elementary classes, independence relations, and module theory that are used in this paper.

\subsection{Abstract elementary classes} We briefly present the notions of abstract elementary classes that are used in this paper. These are further studied in \cite[\S 4 - 8]{baldwinbook09} and  \cite[\S 2, \S 4.4]{ramibook}.  An introduction from an algebraic perspective is given in \cite[\S 2]{maztor}.

Abstract elementary classes (AECs for short) were introduced  by Shelah in \cite{sh88} to study those classes of structures axiomatized in $L_{\omega_1, \omega}(Q)$. An AEC is a pair $\K=(K \lea)$ where $K$ is a class of structures and $\lea$ is a partial order on $K$ extending the substructure relation such that $\K$ is closed under direct limits and satisfies the coherence property and an instance of the Downward L\"{o}wenheim-Skolem theorem. The reader can consult the definition in \cite[4.1]{baldwinbook09}.

Given a model $M$, we will write $|M|$ for its underlying set and $\| M \|$ for its cardinality.  Given $\lambda$ a cardinal and $\K$ an AEC, we denote by $\K_{\lambda}$ the models in $\K$ of cardinality $\lambda$. Moreover, if we write ``$f: M \to N$", we assume that
$f$ is a $\K$-embedding, i.e., $f: M \cong f[M]$ and $f[M] \lea N$.
In particular, $\K$-embeddings are always monomorphisms.

Shelah introduced a notion of semantic type in \cite{sh300}. Following \cite{grossberg2002}, we call
these semantic types Galois-types. Given $(\bb, A, N)$, where $N \in \K$, $A \subseteq |N|$, and $\bb$ is a sequence in $N$, the \emph{Galois-type of $\bb$ over $A$ in $N$}, denoted by $\gtp_{\K} (\bb / A; N)$, is the equivalence class of $(\bb, A, N)$ modulo $E^\K$; $E^\K$ is the transitive closure of $E_{\text{at}}^{\K}$ where $(\bb_1, A_1, N_1)E_{\text{at}}^{\K} (\bb_2, A_2, N_2)$ if $A
:= A_1 = A_2$, and there exist $\K$-embeddings $f_\ell : N_\ell \xrightarrow[A]{} N$ for $\ell \in \{ 1, 2\}$ such that
$f_1 (\bb_1) = f_2 (\bb_2)$ and $N \in \K$. Given $p=\gtp_{\K}(\bb/A; N)$ and $C \subseteq A$, let $p\upharpoonright{C}= [(\bb, C, N)]_{E^\K}$.

If $M \in K$ and $\alpha$ is an ordinal, let $\gS^\alpha_{\K}(M)= \{  \gtp_{\K}(\bb / M; N) : M
\leq_{\K} N\in \K \text{ and } \bb \in N^\alpha\} $. When $\alpha =1$, we write $\gS_{\K}(M)$ instead of $\gS^1_{\K}(M)$. We let $\gS^{< \infty}_{\K}(M) = \bigcup_{\alpha \in OR} \gS^\alpha_{\K}(M)$.

Since Galois-types are equivalence classes, they might not be determined by their finite restrictions. We say that $\K$ is fully \emph{$(< \aleph_0)$-tame} if for any $M \in \K$ and $p \neq q \in \gS^{< \infty}(M)$,  there is $A \subseteq |M|$ such that $|A |< \aleph_0$ and $p\upharpoonright{A} \neq q\upharpoonright{A}$. Tameness was isolated by Grossberg and VanDieren in \cite{tamenessone}.

We now introduce the main notion of this paper.

\begin{defin} An AEC $\K$ is \emph{$\lambda$-stable}  if for any $M \in
\K_\lambda$, $| \gS_{\K}(M) | \leq \lambda$.  
\end{defin}

Recall that a model $M$ is \emph{universal over} $N$ if and only if $\| N\|= \| M\|=\lambda $ and for every $N^* \in \K_{\lambda}$ such that $N \lea N^*$, there is $f: N^* \xrightarrow[N]{} M$. Let us recall the notion of limit model.

\begin{defin}\label{limit}
Let $\lambda$ be an infinite cardinal and $\alpha < \lambda^+$ be a limit ordinal.  $M$ is a \emph{$(\lambda,
\alpha)$-limit model over} $N$ if and only if there is $\{ M_i : i <
\alpha\}\subseteq \K_\lambda$ an increasing continuous chain such
that:
\begin{enumerate}
\item $M_0 =N$.
\item $M= \bigcup_{i < \alpha} M_i$.
\item $M_{i+1}$ is universal over $M_i$ for each $i <
\alpha$.
\end{enumerate}

$M$ is a $(\lambda, \alpha)$-limit model if there is $N \in
\K_\lambda$ such that $M$ is a $(\lambda, \alpha)$-limit model over
$N$. $M$ is a $\lambda$-limit model if there is a limit ordinal
$\alpha < \lambda^+$ such that $M$  is a $(\lambda,
\alpha)$-limit model.

\end{defin}

 We say that $\K$ has \emph{uniqueness of limit models of cardinality $\lambda$} if $\K$ has $\lambda$-limit models and if any two $\lambda$-limit models are isomorphic. We introduce the notion of superstability for AECs.

\begin{defin}
$\K$ is a \emph{superstable} AEC if and only if $\K$ has uniqueness of limit models on a tail of cardinals.
\end{defin}

\begin{remark}
In \cite[1.3]{grva} and \cite{vaseyt} was shown that for AECs that have amalgamation, joint embedding, no maximal models and are tame, the definition above  is equivalent to every other definition of superstability considered in the context of AECs. In particular for a complete first-order theory $T$, $(Mod(T), \preceq)$ is superstable if and only if $T$ is $\lambda$-stable for every $\lambda \geq 2^{|T|}$. 
\end{remark} 

Finally, recall that a model $M \in \K$ is a \emph{universal model in
$\K_\lambda$} if $M \in \K_\lambda$ and if given any $N \in \K_\lambda$, there is a $\K$-embedding $f: N \to M$. We say that $\K$ has a universal model of cardinality $\lambda$ if there is a universal model in $\K_\lambda$. It is well-known that if $\K$ is an AEC with the joint embedding property and $M$ is a $\lambda$-limit model, then $M$ is universal in $\K_\lambda$.

\subsection{Independence relations} We recall the basic properties of independence relations on arbitrary categories. These were introduced and studied in detail in \cite{lrv1}.

\begin{defin}[{ \cite[3.4]{lrv1}}] An independence relation on a category $\mathcal{C}$ is a set $\dnf$ of commutative squares such that for any commutative diagram:

\[
  \xymatrix@=3pc{
    & & E \\
    B \ar[r]^{g_1}\ar@/^/[rru]^{h_1} & D \ar[ru]^t  & \\
    A \ar [u]^{f_1} \ar[r]_{f_2} & C \ar[u]_{g_2} \ar@/_/[ruu]_{h_2} &
  }
\]

we have that $(f_1, f_2, g_1, g_2) \in \dnf$ if and only if $(f_1, f_2, h_1, h_2) \in \dnf$.

\end{defin} 

We will be particularly interested in weakly stable independence relations. Recall that an independence relation $\dnf$ is \emph{weakly stable} if it satisfies: symmetry \cite[3.9]{lrv1},  existence \cite[3.10]{lrv1}, uniqueness \cite[3.13]{lrv1}, and transitivity \cite[3.15]{lrv1}.

They also introduced the notion of a stable independence relation for any category $\mathcal{C}$ in \cite[3.24]{lrv1}. As the definition is long and we will only study independence relations on AECs, we introduce the definition for AECs instead. For an AEC $\K$, an indepedence relation $\dnf$ is \emph{stable} if it is weakly stable and satisfies local character \cite[8.6]{lrv1} and the witness property \cite[8.7]{lrv1}.

\subsection{Module Theory}  We succinctly introduce the notions from module theory that are used in this paper. These are further studied in \cite{prest}.

All rings considered in this paper are associative with unity. In the rest of the paper, if we mention that $R$ is a ring, we are assuming that it is associative with unity. All the classes studied in this paper have as their language the standard language of modules, i.e., for a ring $R$ we take $L_{R}= \{0, +,-\} \cup \{ r\cdot  : r \in R \}$. A formula $\phi$ is a positive primitive formula ($pp$-formula for short), if $\phi$ is an existentially quantified finite system of linear equations. Given $\bar{b} \in M^{< \infty}$ and $M \subseteq N$, the $pp$-type of $\bar{b}$ over $M$ in $N$, denoted by $\pp(\bar{b}/M , N)$, is the set of $pp$-formulas with parameters in $M$ that hold for $\bar{b}$ in $N$.

Given $M$ and $N$ $R$-modules, $M$ is a \emph{pure submodule} of $N$, denoted by $M \leq_{p} N$, if and only if $M$ is a submodule of $N$  and for every $\bar{a} \in M^{< \omega}$, $pp(\bar{a}/ \emptyset, M)= pp(\bar{a}/\emptyset , N)$. Moreover, $f: M \to N$ is a \emph{pure epimorphism} if $f$ is an epimorphism and the kernel of $f$ is a pure submodule of $M$.

Recall that a module $M$ is \emph{pure-injective} if for every $N$, if $M$ is a pure submodule of $N$, then $M$ is a direct summand of $N$. Given a module $M$, the \emph{pure-injective envelope of $M$}, denoted by
$PE(M)$, is a pure-injective module such that $M \leq_{p}
PE(M)$ and it is minimum with respect to this property. Its existence
follows from \cite[3.6]{ziegler} and the fact that every module can be
embedded into a pure-injective module.

 The following Schr\"{o}der-Bernstein property of pure-injective modules will be useful.

\begin{fact}[{\cite[2.5]{gks}}]\label{ipi}
 Let $M, N$ be pure-injective modules. If there are $f: M \to N$ a pure embedding and $g: N \to M$ a pure embedding, then $M$ and $N$ are isomorphic. 
\end{fact}

$M$ is \emph{$\Sigma$-pure-injective} if $M^{(\aleph_0)}$ is pure-injective where $M^{(\aleph_0)}$ denotes the countable direct sum of $M$. The next four properties of $\Sigma$-pure-injective modules will be useful. The first three bullet points follow from \cite[2.11]{prest}.

\begin{fact}\label{spi}\
\begin{itemize}
\item If $N$ is $\Sigma$-pure-injective, then $N$ is pure-injective.
\item If $N$ is $\Sigma$-pure-injective and $M \leq_{p} N$, then $M$ is $\Sigma$-pure-injective.
\item If $N$ is $\Sigma$-pure-injective and $M$ is elementary equivalent to $N$, then $M$ is $\Sigma$-pure-injective.
\item (\cite[3.2]{prest}) If $N$ is $\Sigma$-pure-injective, then $(Mod(Th(N)), \leq_p)$ is $\lambda$-stable for every $\lambda \geq |Th(N)|$. For the models of the theory of $N$, being a pure submodule is the same as being an elementary submodule by $pp$-quantifier elimination. 
\end{itemize}
\end{fact}



\section{Classes closed under pure-injective envelopes}

In this section we study classes closed under direct sums and pure-injective envelopes. We show that they are always stable and we give an algebraic characterization of when they are superstable. 

\begin{hypothesis}\label{hyp1}
Let $\K= (K , \leq_{p})$ be an AEC with $K \subseteq R\text{-Mod}$ for a fixed ring $R$ such that: 
\begin{enumerate}
\item $K$ is closed under direct sums.
\item $K$ is closed under pure-injective envelopes, i.e., if $M \in K$, then $PE(M) \in K$.
\end{enumerate}
\end{hypothesis}

\begin{remark}
Most of the results in this section assume the above hypothesis, but not all of them. We will explicitly mention when we assume the hypothesis.
\end{remark}

Below we give some examples of classes of modules satisfying Hypothesis \ref{hyp1}. 

\begin{example}\label{ex1}\
\begin{enumerate}
\item $(R\text{-AbsP}, \leq_{p})$ where $R\text{-AbsP}$ is the class of absolutely pure $R$-modules. A module $M$ is absolutely pure if it is pure in every module containing it. It is an AEC because being a pure submodule is tested by finite tuples and because it is closed under pure submodules \cite[2.3.5]{prest09}. Closure under direct sums follows from \cite[2.3.5]{prest09} and closure under pure-injective envelopes follows from \cite[4.3.12]{prest09}.

\item $(R\text{-l-inj}, \leq_p)$ where $R\text{-l-inj}$ is the class of locally injective $R$-modules (also called finitely injective modules). A module $M$ is locally injective if given $\bar{a} \in M^{< \omega}$ there is an injective submodule of $M$ containing $\bar{a}$.  It is an AEC because we only test for finite tuples and because the cardinality of the injective envelope of a finite tuple is bounded by $2^{|R| + \aleph_0}$ \cite[Theo 1]{eklof}. Closure under direct sums is clear and closure under pure-injective envelopes follows from the fact that locally injective modules are absolutely pure \cite[3.1]{rara} and \cite[4.3.12]{prest09}.

\item $(R\text{-l-pi}, \leq_p)$ where $R\text{-l-pi}$ is the class of locally pure-injective $R$-modules. A module $M$ is locally pure-injective if given $\bar{a} \in M^{< \omega}$ there is a pure-injective pure submodule of $M$ containing $\bar{a}$. It is an AEC because we only test for finite tuples and because the cardinality of the pure-injective envelope of a finite tuple is bounded by $2^{|R| + \aleph_0}$ \cite[3.11]{ziegler}. Closure under direct sums and pure-injective envelopes follow from \cite[2.4]{zimm2}.

\item $(\text{RTF}, \leq_p)$ where  \text{RTF} is the class of reduced torsion-free abelian groups. A group $G$ is reduced if it does not have non-trivial divisible subgroups. It is an AEC because the intersection of a torsion-free divisible subgroup with a torsion-free pure subgroup is a divisible subgroup and because it is closed under pure subgroups. Closure under direct sums is easy to check, while closure under pure-injective envelopes follows from \cite[6.4.3]{fucb}. 

\item $(R\text{-Flat}, \leq_{p})$ where $R\text{-Flat}$ is the class of flat $R$-modules under the additional assumption that the pure-injective envelope of every flat modules is flat.\footnote{These rings were introduced in \cite{roth} and this class was studied in detail in \cite[\S 3]{mazf}.} It is an AEC because it is closed under direct limits and pure submodules. Closure under direct sums is easy to check and we are assuming closure under pure-injective envelopes. 
\item $(\chi, \leq_{p})$ where $\chi$ is a definable category of modules in the sense of \cite[\S 3.4]{prest09}. A class of modules is definable if it is closed under direct products, direct limits and pure submodules. It is an AEC because it is closed under direct limits and pure submodules. Closure under pure-injective envelopes follows from \cite[4.3.21]{prest09}. 
\end{enumerate}
\end{example}

\begin{remark}
It is worth mentioning that none of the above examples are first-order axiomatizable with the exception of the last one.
\end{remark}

\subsection{Stability} We begin by showing some structural properties of the classes satisfying Hypothesis \ref{hyp1}. The argument for the amalgamation property is due to T.G. Kucera. 

\begin{lemma}\label{stru1}
If $\K$ satisfies Hypothesis \ref{hyp1}, then $\K$ has joint embedding, amalgamation, no maximal models and $|R| + \aleph_0 \leq \LS(\K)$.  
\end{lemma}
\begin{proof}
Joint embedding and no maximal models follow directly from closure under direct sums. So we show the amalgamation property. 

Let $M \leq_p N_1, N_2$ be models of $K$. By minimality of the pure-injective envelope we obtain that $PE(M) \leq_p PE(N_1), PE(N_2)$ and observe that all of these models are in $K$ by closure under pure-injective envelopes.

Let $L := PE(N_1) \oplus PE(N_2)$ which is in $K$ by closure under direct sums. Now, as $PE(M)$ is pure-injective, there are $N_1'$ and $N_2'$ such that $PE(N_i)= PE(M) \oplus N_i'$ for $i \in \{ 1, 2\}$. Hence, $L = (PE(M) \oplus N_1') \oplus (PE(M) \oplus N_2')$. Define $f: N_1 \to L$ by 
$f(m + n_1)=(m, n_1, m, 0)$ for $m \in PE(M)$ and $n_1 \in N_1'$ and 
$g: N_2 \to L$ by $g(m + n_2)=(m, 0, m, n_{2})$ for $m \in PE(M)$ and $n_2 \in N_2'$.
One can show that $f, g$ are pure embeddings such that 
$f\rest{M}= g\rest{M}$.
\end{proof}

The next abstraction is a first step toward a general solution to Question \ref{mainq}. We thank an anonymous referee for suggesting this approach. 

\begin{defin}  The Galois-types in $\K$ are \emph{$pp$-syntactic} if for every  $M, N_1, N_2 \in \K$,  $M \lea  N_1, N_2$, $\bar{b}_{1}
\in  N_1^{<\infty}$  and $\bar{b}_{2} \in N_2^{<\infty}$ we have that:
 \[ \gtp_\K(\bar{b}_{1}/M; N_1) = \gtp_\K(\bar{b}_{2}/M; N_2) \text{ if
and
only if } \pp(\bar{b}_{1}/M , N_1) = \pp(\bar{b}_{2}/M, N_2).\]


\end{defin}

\begin{remark}
It is straightforward to show that if Galois-types in $\K$ are \emph{$pp$-syntactic}, then $\K$ is fully $(< \aleph_0)$-tame.
\end{remark}

Our main result regarding classes where Galois-types are $pp$-syntactic is the following.

\begin{theorem}\label{gstat}
Assume $\K= (K , \leq_{p})$ is an AEC with $K \subseteq R\text{-Mod}$ for a fixed ring $R$.  If the Galois-types in $\K$ are $pp$-syntactic, then $\K$ is $\lambda$-stable for every $\lambda \geq \LS(\K)$ such that $\lambda^{|R| + \aleph_0}=\lambda$.
\end{theorem}
\begin{proof}
Let $\lambda \geq \LS(\K)$ such that $\lambda^{|R| + \aleph_0}=\lambda$ and $M \in \K_\lambda$. Let $\{ p_i : i < \kappa \}$ be an enumeration without repetitions of $\gS(M)$. For every $i < \kappa$ fix a pair $(a_i, N_i)$ such that $p_i=\gtp(a_i/M; N_i)$. Let $\Delta \subseteq \kappa$ such that $| \Delta | \leq 2^{|R| + \aleph_0}$ and for every $i < \kappa$ there is a $j \in \Delta$ such that $N_i$ is elementarily equivalent to $N_j$. This is possible because there are at most $2^{|R| + \aleph_0}$ complete theories over $R$.

Let $\Phi: \kappa \to \bigcup_{j \in \Delta} S_{pp}^{Th(N_j)}(M)$ be such that $\Phi(i)= pp(a_i/M, N_i)$. By the choice of $\Delta$ and the hypothesis that Galois-types in $\K$ are $pp$-syntactic we have that $\Phi$ is a well-defined injective function, so $\kappa  \leq  |\bigcup_{j \in \Delta} S_{pp}^{Th(N_j)}(M)|$. Observe that for every $j \in \Delta$ we have that $|S_{pp}^{Th(N_j)}(M)|=|S^{Th(N_j)}(M)|$ by $pp$-quantifier elimination (see \cite[\S 2.4]{prest}). Hence $|\bigcup_{j \in \Delta} S_{pp}^{Th(N_j)}(M)|= |\bigcup_{j \in \Delta} S^{Th(N_j)}(M)|\leq \sum_{j \in \Delta} | S^{Th(N_j)}(M)|$. Since every complete first-order theory of modules is $\lambda$-stable if $\lambda^{|R| + \aleph_0}=\lambda$ by \cite[3.1]{prest} and $|\Delta | \leq 2^{|R| + \aleph_0}$, we have that $\sum_{j \in \Delta} | S^{Th(N_j)}(M)| \leq \lambda$. Hence, $\kappa \leq \lambda$. Therefore, $\K$ is $\lambda$-stable.  \end{proof} 

We show that Galois-types are $pp$-syntactic for classes satisfying Hypothesis \ref{hyp1}.. The result is similar to \cite[3.14]{kuma}, but the argument given there cannot be applied in this setting. A similar argument than that of \cite[4.4]{mazf} works in the more general setting of classes satisfying Hypothesis \ref{hyp1}.

\begin{lemma}\label{pp=gtp} Assume $\K$ satisfies Hypothesis \ref{hyp1}. Then Galois-types in $\K$ are $pp$-syntactic.
\end{lemma}
\begin{proof}
The forward direction is trivial so we show the backward direction. As $K$ has the amalgamation property we may assume that $N_1=N_2$ and since $K$ is closed under pure-injective envelopes we may assume that $N_1 = N_2$ is pure-injective. Let $N = N_1 =N_2$. Then by \cite[3.6]{ziegler} there is \[ h: H^{N}(M \cup \{\bar{b}_1 \}) \cong_{M} H^{N}(M \cup \{ \bar{b}_2 \}) \] with $h(\bar{b}_1)= \bar{b}_2$ where $H^{N}(M \cup \{\bar{b}_i\})$ denotes the pure-injective envelope of $M \cup \{\bar{b}_i\}$ inside $N$ for every $i \in \{ 1, 2 \}$.

As it might be the case that $H^{N}(M \cup \{\bar{b}_1 \})$ and $H^{N}(M \cup \{ \bar{b}_2 \})$ are not in $K$, we can not simply apply the amalgamation property and instead we have to do a similar argument to that of Lemma \ref{stru1}.

Given $ i \in \{1, 2\}$, $H^{N}(M \cup \{\bar{b}_i\})$ is pure-injective so there is $L_i$ such that $N=H^{N}(M \cup \{\bar{b}_i\}) \oplus L_i$.  Let $L= (H^{N}(M \cup \{\bar{b}_2\}) \oplus L_1) \oplus (H^{N}(M \cup \{\bar{b}_2\}) \oplus L_2)$. Observe that $L \in K$ by closure under direct sums. Define $f: N=H^{N}(M \cup \{\bar{b}_1\}) \oplus L_1 \to L$ by $f(s + l_1)=(h(s), l_1, h(s), 0)$ for $s \in H^{N}(M \cup \{\bar{b}_1\})$ and $l_1 \in L_1$. Define $g: N=H^{N}(M \cup \{\bar{b}_2\}) \oplus L_2 \to L$ by $g(q + l_2)=(q, 0, q, l_2)$ for $q \in H^{N}(M \cup \{\bar{b}_2\})$ and $l_2 \in L_2$. It is easy to check that $L, f, g$ witness that $\gtp(\bar{b}_{1}/M; N_1) = \gtp(\bar{b}_{2}/M; N_2)$. \end{proof} 

An  immediate corollary is that the classes satisfying Hypthesis \ref{hyp1} are tame.

\begin{cor}\label{fullt}
If $\K$ satisfies Hypothesis \ref{hyp1}, then $\K$ is fully $(< \aleph_0)$-tame. 
\end{cor}

The next results follows from Theorem \ref{gstat} and Lemma \ref{pp=gtp}.

\begin{theorem}\label{stat1}
Assume $\K$ satisfies Hypothesis \ref{hyp1} and $\lambda \geq \LS(\K)$.  If $\lambda^{|R| + \aleph_0}=\lambda$, then $\K$ is $\lambda$-stable.
\end{theorem}

Then from \cite[3.20]{kuma} we can conclude the existence of universal models.

\begin{cor}\label{universal2} Assume $\K$ satisfies Hypothesis \ref{hyp1} and $\lambda \geq \LS(\K)$.
 If $\lambda^{|R| + \aleph_0}=\lambda$ or $\forall \mu < \lambda( \mu^{|R| + \aleph_0} <
\lambda)$, then $\K$ has a universal model of cardinality $\lambda$. 
\end{cor}

\subsection{Limit models and superstability} Since $K$ has joint embedding, amalgamation and no maximal models, it follows from \cite[\S II.1.16]{shelahaecbook} that $\K$ has a $(\lambda, \alpha)$-limit model if $\lambda^{|R| + \aleph_0}=\lambda$, $\lambda \geq \LS(\K)$ and $\alpha < \lambda^+$ is a limit ordinal. We characterize limit models with chains of big cofinality. This extends \cite[4.5]{kuma} and \cite[4.9]{mazf} to any class satisfying Hypothesis \ref{hyp1}.

\begin{lemma}\label{bigpi}
Assume $\K$ satisfies Hypothesis \ref{hyp1} and $\lambda\geq \LS(\K)^+$. If $M$ is a $(\lambda,
\alpha)$-limit model and $\cof(\alpha)\geq \rc$, then $M$ is
pure-injective.
\end{lemma}
\begin{proof}
Fix $\{ M_i : i < \alpha\}$ a witness to the fact that $M$ is a
$(\lambda, \alpha)$-limit model. We show that every $p(x)$  $ M$-consistent $pp$-type over $A \subseteq M$ with $|A|
\leq |R| + \aleph_0$ is realized in $M$.\footnote{ For an incomplete
theory $T$ we say that a  pp-type $p(x)$ over $A \subseteq M$ is
$M$-consistent if it is realized in an elementary extension of $M$.}  This enough to show that $M$ is pure-injective by \cite[2.8]{prest}.

Observe that $p$ is a $PE(M)$-consistent $pp$-type as $M \preceq PE(M)$. Since $PE(M)$ is pure-injective, it is saturated for $pp$-types \cite[2.8]{prest}, so there is $a \in PE(M)$ realizing $p$. As  $\cof(\alpha)\geq \rc$, there is $i < \alpha$ such that $A \subseteq M_i$. Applying downward L\"{o}wenheim-Skolem to $M_i \cup \{ a \}$ in $PE(M)$ we obtain $N \in \K_\lambda$ with $M_i \leq_p N$ and $a \in N$. Then there is $f: N \xrightarrow[M_i]{} M$ because $M_{i + 1}$ is universal over $M_i$.  Hence $f(a) \in M$ realizes $p$.
\end{proof}

Since $K$ is closed under direct sums, the usual argument \cite[4.9]{kuma} can be use to characterize limit models of countable cofinality. 

\begin{lemma}\label{countablelim}  Assume $\K$ satisfies Hypothesis \ref{hyp1} and $\lambda\geq
\LS(\K)^+$. If $M$ is
a $(\lambda, \omega)$-limit model and  $N$ is a
$(\lambda, \rc)$-limit model, then $M$ is isomorphic to $N^{(\aleph_0)}$.
\end{lemma}

Moreover, any two limit models of $\K$ are elementarily equivalent. The proof is similar to that of  \cite[4.3]{kuma} so we omit it. 

\begin{lemma}\label{limeq} Assume $\K$ satisfies Hypothesis \ref{hyp1}.
If $M, N$ are limit models of $\K$, then $M$ and $N$ are elementary equivalent.
\end{lemma}

\begin{remark} Lemma \ref{bigpi} and Lemma \ref{countablelim} describe limit models with \emph{long} chains and the limit model with  \emph{the  shortest} chain under Hypothesis \ref{hyp1}. We do not know how the other limit models look like except from the fact that they are elementarily equivalent to the ones we understand by Lemma \ref{limeq}. Since this is all we need to characterize superstability, we do not explore this any further in this paper. Nevertheless, we think that understanding the other limit models could help better understand the classes satisfying Hypothesis \ref{hyp1}.

\end{remark}

Due to Lemma \ref{limeq}, it makes sense to introduce the following first-order theory:

\begin{nota}\label{notat}
For $\K$ satisfying Hypothesis \ref{hyp1}, let $\tilde{M}_\K$ be the $(2^{\LS(\K)}, \omega)$-limit model of $\K$ and $\tilde{T}_\K=Th(\tilde{M}_\K)$.
\end{nota}

In \cite[\S 4.1]{maz1} a similar theory, called $\tilde{T}$ there, was introduced.  There it was shown that there was a very close relation between the AEC $\K$ and $\tilde{T}_\K$. We do not think that this is the case when $\K$ satisfies Hypothesis \ref{hyp1} and  is not first-order axiomatizable. We think that this is the case because there can be models of $\tilde{T}_\K$ that are not in $\K$. Nevertheless, stability transfers from $\tilde{T}_{\K}$  to $\K$. As the proof is similar to that of \cite[4.9]{maz1} we omit it.

\begin{lemma}\label{implyst}
Assume $\K$ satisfies Hypothesis \ref{hyp1}  and let $\lambda \geq \LS(\K)$. If $\tilde{T}_{\K}$ is $\lambda$-stable, then $\K$ is $\lambda$-stable.
\end{lemma}

\begin{remark} In \cite[4.9]{maz1}, it is shown that the converse is true if $\K$ is first-order axiomatizable.  We do not think that the converse is true in this more general setting, but we do not have a counterexample. 
\end{remark}


We characterize superstability for classes satisfying Hypothesis \ref{hyp1}. The next result extends \cite[4.26]{maz1} to classes not necessarily axiomatizable by a first-order theory and \cite[4.12]{mazf} to a different class than that of Example 3.2.(5).

\begin{theorem}\label{main} Assume $\K$ satisfies Hypothesis \ref{hyp1}. The following are equivalent.
\begin{enumerate}
\item  $\K$ is superstable.
\item There is a $\lambda \geq \LS(\K)^+$ such that $\K$ has uniqueness of limit models of cardinality $\lambda$.
\item Every limit model in $\K$ is $\Sigma$-pure-injective. 
\item Every model in $\K$ is pure-injective.
\item For every $\lambda \geq \LS(\K)$, $\K$ has uniqueness of limit models of cardinality $\lambda$.

\end{enumerate}
\end{theorem}
\begin{proof}
(1) $\Rightarrow$ (2) Clear.

(2) $\Rightarrow$ (3)  Let $\lambda \geq \LS(\K)^+$ such that $\K$ has uniqueness of limit models of size $\lambda$. Let $M$ be a $(\lambda, \rc)$-limit model in $\K$. It follows from Lemma \ref{countablelim} that $M^{(\aleph_0)}$ is the $(\lambda, \omega)$-limit model. As $\K$ has uniqueness of limit models of size $\lambda$, we have that $M$ is isomorphic to $M^{(\aleph_0)}$. Since $M$  is pure-injective by Lemma \ref{bigpi}, it follows that $M^{(\aleph_0)}$ is pure-injective. Hence $M$ is $\Sigma$-pure-injective. Since limit models are elementarily equivalent by Lemma \ref{limeq} and $\Sigma$-pure-injectivity is preserved under elementarily equivalence by Fact \ref{spi}, it follows that every limit model is $\Sigma$-pure-injective.

(3) $\Rightarrow$ (4)
Let $N\in \K$ and $N'$ be a $(\|N\|^{|R| + \aleph_0}, \omega)$-limit model, this exists by Theorem \ref{stat1}.  Then there is $f: N \to N'$ a pure embedding by \cite[2.10]{maz}. Since $N'$ is $\Sigma$-pure-injective and $f$ is a pure embedding, it follows from Fact \ref{spi} that $N$ is $\Sigma$-pure-injective. Hence every model in $\K$ is pure-injective.

(4) $\Rightarrow$ (5)   Let $M$ be a $(2^{\LS(\K)}, \omega)$-limit model. By (4) and closure under direct sums we have that $M$ is $\Sigma$-pure-injective, so $Th(M)$  is $\lambda$-stable for every $\lambda \geq  |R| + \aleph_0$ by Fact \ref{spi}.  As $Th(M)=\tilk$ by definition, it follows from Lemma \ref{implyst} that $\K$ is $\lambda$-stable for every $\lambda \geq  \LS(\K)$. Therefore, by \cite[\S II.1.16]{shelahaecbook} there exist a $\lambda$-limit model for every  $\lambda \geq  \LS(\K)$.

Regarding uniqueness, observe that given $M$ and $N$ $\lambda$-limit models, there are $f: M \to N$ and $g: N \to M$ pure embeddings by \cite[2.10]{maz}. Since we have that $M$ and $N$ are pure-injective, it follows from Fact \ref{ipi} that $M$ and $N$ are isomorphic.

(5) $\Rightarrow$ (1) Clear. \end{proof}

\begin{remark} It can also be shown as in \cite[4.26]{maz1} that $\K$ is superstable if and only if there exists $\lambda \geq \LS(\K)^+$ such that $\K$ has a $\Sigma$-pure-injective universal model of cardinality $\lambda$.
\end{remark}

\subsection{Characterizing several classes of rings} We will use the results of the preceding subsection to characterize noetherian rings, pure-semimple rings, Dedekind domains, and fields via superstability.

Recall that a module $M$ is injective if it is a direct summand of every module containing it. The next result will be useful.

\begin{fact}[{\cite[4.4.17]{prest09}}]\label{eno} Let $R$ be a ring. The following are equivalent.
\begin{enumerate}
\item $R$ is left noetherian.
\item The class of absolutely pure left $R$-modules is the same as the class of injective left $R$-modules.
\item Every direct sum of injective left $R$-modules is injective.
\end{enumerate}
\end{fact}

We begin by giving two new characterizations of noetherian rings. The equivalence between (1) and (2) extends \cite[4.30]{maz1}. Recall that $R\text{-AbsP}$ is the class of absolutely pure $R$-modules and that  $R\text{-l-inj}$ is the class of locally injective $R$-modules, these were introduced in Example \ref{ex1}.

\begin{theorem}\label{absss}
Let $R$ be a ring. The following are equivalent.

\begin{enumerate}
\item $R$ is left noetherian.
\item $(R\text{-AbsP}, \leq_{p})$ is superstable.
\item $(R\text{-l-inj}, \leq_p)$ is superstable.
\end{enumerate}

\end{theorem} 
\begin{proof}
Recall that absolutely pure modules and locally injective modules satisfy Hypothesis \ref{hyp1}, so we can use the results from the previous subsection. More  precisely, we use Theorem \ref{main}.(4) to show the equivalences.

(1) $\Rightarrow$ (2) If $R$ is noetherian, then every absolutely pure module is injective by Fact \ref{eno}. Hence, every absolutely pure module is pure-injective. So the result follows from Theorem \ref{main}.

(2) $\Rightarrow$ (3) Every locally injective module is absolutely pure by \cite[3.1]{rara}. Then it follows that every locally injective module is pure-injective by (2). Hence, the class of locally injective $R$-modules is superstable. 

(3) $\Rightarrow$ (1) We show that the direct sum of injective modules is injective, this is enough by Fact \ref{eno}. Let $\{ M_i : i \in I \}$ be a family of injective modules. As they are all locally injective, we have that $\bigoplus_{i \in I} M_i$ is locally injective. Moreover, as $(R\text{-l-inj}, \leq_p)$  is superstable, we have that  $\bigoplus_{i \in I} M_i$ is also pure-injective by Theorem \ref{main}. Recall that locally injective modules are absolutely pure, so $\bigoplus_{i \in I} M_i$ is absolutely pure and pure-injective. Therefore,  $\oplus_{i \in I} M_i$ is injective. Hence $R$ is noetherian. \end{proof} 

We use the above result to study the class of injective $R$-modules with pure embeddings, we will denote it by $(R\text{-Inj}, \leq_{p})$.

\begin{cor}\label{inj}
Let $R$ be a ring. $(R\text{-Inj}, \leq_{p})$ is an AEC if and only if $R$ is left noetherian.  Moreover, if $R$ is left noetherian, then $(R\text{-Inj}, \leq_{p})$ is a superstable AEC.
\end{cor}
\begin{proof}
If $(R\text{-Inj}, \leq_{p})$ is an AEC then the direct sum of injective modules is an injective module because injective modules are closed under finite direct sums. Hence $R$ is left noetherian. On the other hand, if $R$ is left noetherian, then injective modules are the same as absolutely pure modules by Fact \ref{eno}. Hence $(R\text{-Inj}, \leq_{p})$ is an AEC.

The moreover part follows directly from Theorem \ref{absss}.
\end{proof}

The next corollary shows a connection between being \emph{good} in the stability hierachy and being \emph{good} in the axiomatizability hierachy. 

\begin{cor}\label{good} Let $R$ be a ring.
\begin{enumerate}
\item If $(R\text{-AbsP}, \leq_{p})$ is superstable, then the class of absolutely pure left $R$-modules is first-order axiomatizable.
\item If $(R\text{-l-inj}, \leq_{p})$ is superstable, then the class of locally injective left $R$-modules is first-order axiomatizable.
\end{enumerate}
\end{cor}
\begin{proof}\
\begin{enumerate}
\item Since $(R\text{-AbsP}, \leq_{p})$ is superstable, then by Theorem \ref{absss} $R$ is left noetherian. Then $R$ is left coherent, so it follows from \cite[3.4.24]{prest09} that absolutely pure modules are first-order axiomatizable
\item The proof is similar to that of (1), using that if $R$ is noetherian then the class of absolutely pure modules is the same as the class of locally injective modules. 
\end{enumerate}
 \end{proof}

We turn our attention to pure-semisimple rings.  A ring is \emph{pure-semisimple} if and only if every $R$-module is pure-injective. These have been thoroughly studied \cite{cha, auslander1, auslander, simson77, zim, simson81, prest84 , simson00, prest09, maz1}. Recall that $R\text{-l-pi}$ is the class of locally pure-injective $R$-modules, these were introduced in Example \ref{ex1}. The equivalence between (1) and (2) of the next assertion was obtained in \cite[4.28]{maz1}.

\begin{theorem}\label{psss}
Let $R$ be a ring. The following are equivalent.

\begin{enumerate}
\item $R$ is left pure-semisimple.
\item $(R\text{-Mod}, \leq_{p})$ is superstable.
\item $(R\text{-l-pi}, \leq_p)$  is superstable. 
\end{enumerate}
\end{theorem}
\begin{proof} Recall that $R$-modules and locally pure-injective $R$-modules satisfy Hypothesis \ref{hyp1}. We use Theorem \ref{main}.(4) to show the equivalences. The equivalence between (1) and (2) and the direction (2) to (3) are straightforward. We show (3) to (1).

Let $M$ be an $R$-module, then $PE(M)$ is locally pure-injective and $M \leq_p PE(M)$. Observe that $PE(M)^{(\aleph_0)}$ is locally pure-injective. Then $PE(M)^{(\aleph_0)}$ is pure-injective by hypothesis (3), so $PE(M)$ is $\Sigma$-pure-injective. Hence, $M$ is pure-injective by Fact \ref{spi}. Therefore, $R$ is left pure-semisimple.  \end{proof}

We can obtain an analogous result to Corollary \ref{inj} by substituting the class of injective modules by that of pure-injective modules. We denote by $R\text{-pi}$ the class of pure-injective $R$-modules.

\begin{cor}
Let $R$ be a ring. $(R\text{-pi}, \leq_{p})$ is an AEC if and only if $R$ is left pure-semisimple.  Moreover, if $R$ is left pure-semisimple, then  $(R\text{-pi}, \leq_{p})$ is a superstable AEC.
\end{cor}
\begin{proof}
If $(R\text{-pi}, \leq_{p})$ is an AEC, then $M^{(\aleph_0)}=\bigcup _{n < \omega} M^{n}$ is pure-injective for every pure-injective module $M$ as pure-injective modules are closed under finite direct sums. So every pure-injective module is $\Sigma$-pure-injective. Then doing an argument similar to that of the previous result, one can show that $R$ is left pure-semisimple. On the other hand, if $R$ is left pure-semisimple, then all modules are pure-injective. Hence  $(R\text{-pi}, \leq_{p})$ is an AEC. 

The moreover part follows directly from Theorem \ref{psss}.
\end{proof}

We also get a relation between being \emph{good} in the stability hierarchy and being \emph{good} in the axiomatizability hierarchy for locally pure-injective modules. 

\begin{cor}\label{good2} Let $R$ be ring. If $(R\text{-l-pi}, \leq_{p})$ is superstable, then the class of locally pure-injective left $R$-modules is the same as the class of left $R$-modules. So clearly, first-order axiomatizable.
\end{cor}
\begin{proof}
Since $(R\text{-l-pi}, \leq_{p})$ is superstable, then by Theorem \ref{psss} $R$ is left pure-semisimple. Hence, every $R$-module is pure-injective, so in particular locally pure-injective.
\end{proof}

Corollaries \ref{good} and \ref{good2} may suggest that given an AEC of modules satisfying Hypothesis \ref{hyp1}, it follows that if the class is superstable, then the class is first-order axiomatizable. This is not the case as witnessed by the next example.

\begin{example}
 It was shown in \cite[3.15]{mazf} that $(R\text{-Flat}, \leq_{p})$ is superstable if and only if $R$ is left perfect. It is known \cite[Theo 4]{eksa} that the class of flat left $R$-modules is first-order axiomatizable if and only if $R$ is right coherent. Therefore, the ring $R$ described in \cite[3.3]{roth} is such that $(R\text{-Flat}, \leq_{p})$ satisfies Hypothesis \ref{hyp1},  $(R\text{-Flat}, \leq_{p})$ is superstable and $R\text{-Flat}$ is not first-order axiomatizable.
\end{example}

As mentioned in the introduction, the main focus of the paper is Question \ref{mainq}. The results of this section can be used to characterized those rings for which all AECs closed under direct sums are superstable.

\begin{lemma}\label{pss} 
Let $R$ be a ring. The following are equivalent. 
\begin{enumerate}
\item $R$ is left pure-semisimple.
\item Every AEC $\K= (K , \leq_{p})$ with $K \subseteq R\text{-Mod}$, such that $K$ is closed under direct sums, is superstable.
\end{enumerate}

\end{lemma}
\begin{proof}
The backward direction follows from Theorem \ref{psss} as $(R\text{-Mod}, \leq_p)$ satisfies the hypothesis of (2). We show the forward direction.

Let $\K$ be a class satisfying the hypothesis of (2). Then $\K$ is closed under pure-injective envelopes as every module is pure-injective since we are assuming that the ring is left pure-semisimple. Hence, $\K$ satisfies Hypothesis \ref{hyp1}. Therefore, $\K$ is superstable by Theorem \ref{main}.(4).  \end{proof}

The next well-known ring theoretic result follows from the above lemma, Theorem \ref{absss} and \cite[3.15]{mazf}.

\begin{cor} Assume $R$ is an associative ring with unity.
If $R$ is left pure-semisimple, then $R$ is left noetherian and left perfect. 
\end{cor}

We finish this subsection by applying the technology developed in this section to integral domains. Given an integral domain $R$, we study the class of divisible $R$-modules, denoted by $R\text{-Div}$, and the class torsion-free $R$-modules, denoted by $R\text{-TF}$. A module $M$ is a divisible $R$-module if for every $m \in M$ and $r \neq 0 \in R$, there is $n \in M$ such that $rn = m$. A module $M$ is a torsion-free $R$-module if for every $m \neq 0 \in M$ and every $r\neq 0 \in R$, $rm \neq 0$. It is easy to show that $(R\text{-Div}, \leq_p)$ and $(R\text{-TF}, \leq_p)$  both satisfy Hypothesis \ref{hyp1}, this is the case as they are both definable classes in the sense of Example \ref{ex1}.(6). 

\begin{lemma} Let $R$ be an integral domain.
\begin{enumerate}
\item $R$ is a Dedekind domain if and only if $(R\text{-Div}, \leq_p)$ is superstable.
\item $R$ is a field if and only if $(R\text{-TF}, \leq_p)$ is superstable.
\end{enumerate}
\end{lemma}
\begin{proof}\
\begin{enumerate}
\item $\Rightarrow$: Since $R$ is a Dedekind domain, every divisible $R$-module is injective by \cite[4.24]{rot}. As injective modules are pure-injective, $(R\text{-Div}, \leq_p)$ is superstable by Theorem \ref{main}.

$\Leftarrow$: Recall that a module is $h$-divisible if it is the epimorphic image of an injective module. Therefore, the class of $h$-divisible $R$-modules is contained in the class of divisible $R$-modules. Then every $h$-divisible $R$-module is pure-injective by Theorem \ref{main}. Therefore, $R$ is a Dedekind domain by \cite[2.5]{sal}. 

\item $\Rightarrow$: If $R$ is a field, clearly $R$ is a Pr\"{u}fer domain. So the class of flat modules is the same as the class of torsion-free modules by \cite[4.35]{rot}.  Then $(R\text{-TF}, \leq_p)$ is superstable since $R$ is perfect and by \cite[3.15]{mazf}.

$\Leftarrow$: It follows from Theorem \ref{main} and \cite[2.3]{sal} that $R$ is a Pr\"ufer domain. So, as before, the class of flat modules is the same as the class of torsion-free modules. Then $R$ is left perfect by \cite[3.15]{mazf}. Therefore, $R$ is a field by \cite[2.3]{sal2}. \end{enumerate}
\end{proof} 

The next result follows directly from the above lemma.
\begin{cor} Let $R$ be an integral domain.
$(R\text{-TF}, \leq_p)$ is superstable if and only if $\left( R\text{-TF}, \leq_p \right)$ is $\lambda$-categorical for every $\lambda \geq (|R| + \aleph_0)^+$.
\end{cor}

Finally, we record a couple of results on AECs of abelian groups . The result for torsion-free abelian groups was first obtained in \cite[0.3]{baldwine}.

\begin{cor}\
\begin{enumerate}
\item The AEC of divisible abelian groups with pure embeddings is superstable. 
\item The AECs of torsion-free abelian groups with pure embeddings and reduced torsion-free abelian groups with pure embeddings are strictly stable, i.e., stable not superstable. 
\end{enumerate}
\end{cor}

\section{Classes closed under pure epimorphic images}

In this section we study classes closed under direct sums, pure submodules, and pure epimorphic images. We show that they are always stable. The proof is different to that of the previous section as we first show the existence of a weakly stable independence relation with local character and from it we obtain the stability cardinals.

\begin{hypothesis}\label{hyp2}
Let $\K= ( K, \leq_{p})$ be an AEC with $K \subseteq R\text{-Mod}$ for a fixed ring $R$ such that:
\begin{enumerate}
\item $K$ is closed under direct sums.
\item $K$ is closed under pure submodules.
\item $K$ is closed under pure epimorphic images.
\end{enumerate}
\end{hypothesis}

\begin{remark}
Most of the results in this section assume the above hypothesis, but not all of them. We will explicitly mention when we assume the hypothesis.
\end{remark}

Below we give some examples of classes of modules satisfying Hypothesis \ref{hyp2}.

\begin{example}\label{ex2}
Our main source of examples are $F$-classes. These were introduced in \cite{prz} and studied in detail in \cite{hero}. Let us recall that an $F$-class is a class of modules axiomatizable by formulas of the form:

\[ \forall x( \phi \left(x \right) \to \bigvee_{\psi\left(x \right) \in \Psi} \psi \left(x \right) ). \]

Where $\phi$ is a $pp$-formula with one free variable and $\Psi$ is a collection of $pp$-formulas (possibly infinite) with one free variable such that $ \psi[M] \subseteq \phi[M]$ for every $\psi \in \Psi$ and $M$ an $R$-module and $\Psi$ is closed under finite sums. Recall that $\Psi$ is closed under finite sums if for every $\psi_0, \cdots ,\psi_{n-1} \in \Psi$, $\psi_0 + \cdots + \psi_{n-1} \in \Psi$ where $\psi_1 + \cdots + \psi_{n-1} (x) = \exists y_1 \cdots \exists y_n  ( x= y_1 + \cdots y_n \wedge  (\bigwedge_{k < n} \psi_k(y_k)) )$

It follows from \cite[2.3]{hero} that every $F$-class is closed under direct sums, pure submodules and pure epimorphic images. Moreover, it is clear that $F$-classes with pure embeddings are AECs. Therefore, every $F$-class satisfies Hypothesis \ref{hyp2}.

Some interesting examples of $F$-classes are\footnote{All of these examples are presented in \cite{roth2} and there it is explained why they are $F$-classes.}:
\begin{enumerate}
\item $(R\text{-Flat}, \leq_{p})$ where $R\text{-Flat}$ is the class of flat left $R$-modules. A module $M$ is flat if $(-) \otimes M$ is an exact functor.
\item $(\text{p-grp}, \leq_p)$ where $\text{p-grp}$ is the class of abelian $p$-groups for $p$ a prime number. A group $G$ is a $p$-group if every element $g \neq 0$ has order $p^n$ for some $n \in \mathbb{N}$. 
\item $(\text{Tor}, \leq_p)$ where $\text{Tor}$ is the class of torsion abelian groups. A group $G$ is a torsion group if every element $g \neq 0$ has finite order. 
\item $(\s\text{-Tor}, \leq_p)$ where $\s\text{-Tor}$ is the class of $\s$-torsion $R$-modules in the sense of \cite{maru}. A module $M$ is an $\s$-torsion module if it satisfies:

\[ \forall x ( x=x \to  \bigvee_{\psi(R)=0,\, \psi \in pp\text{-formula} } \psi \left(x\right) ) \]

This model-theoretic description is obtained in \cite[3.6]{roth2}.
\item $(\chi, \leq_{p})$ where $\chi$ is a definable category of modules in the sense of \cite[\S 3.4]{prest09}.
\end{enumerate}
\end{example}

\begin{remark}
It is worth mentioning that none of the above examples are first-order axiomatizable with the exception of the last one.
\end{remark}

\begin{remark}
$(R\text{-AbsP}, \leq_{p})$ and $(\text{RTF}, \leq_p)$ both satisfy Hypothesis \ref{hyp1}, but do not satisfy Hypothesis \ref{hyp2}. If either class satisfied Hypothesis \ref{hyp2}, then they would be first-order axiomatizable by \cite[3.4.7]{prest09}, which we know is not the case.

On the other hand, $(R\text{-Flat}, \leq_{p})$, $(\text{p-grp}, \leq_p)$ and $(\text{Tor}, \leq_p)$ satisfy Hypothesis \ref{hyp2}, but do not satisfy Hypothesis \ref{hyp1}. The case of flat modules is well-known and for torsion groups see \cite[3.1]{maztor}.

Therefore, the classes of modules satisfying Hypothesis \ref{hyp1} are not contained in those satisfying Hypothesis \ref{hyp2} and vice versa. Definable classes satisfy both of the hypothesis, but there are non-definable classes as well (see Example \ref{ex1}.(5)).
\end{remark}

\subsection{Stability} We begin by recalling some important properties of pushouts in the category of $R$-modules with morphisms, we denote this category by $R$-Mod.

\begin{remark} 
\begin{itemize}\
\item Given a span $( f_1: M \to N_1, f_2: M \to N_2 )$ in $R$-Mod, a \emph{pushout} is a triple $(P, g_1, g_2)$ with $g_1 \circ f_1 = g_2 \circ f_2$ that is a solution to the universal property that for every $(Q, h_1, h_2)$ such that $h_1 \circ f_1 = h_2 \circ f_2$, there is a unique $t: P \to Q$ making the following diagram commute:

\[
  \xymatrix@=3pc{
    & & Q \\
    N_1 \ar[r]^{g_1}\ar@/^/[rru]^{h_1} & P \ar[ru]^t \pullbackcorner & \\
    M \ar [u]^{f_1} \ar[r]_{f_2} & N_2 \ar[u]_{g_2} \ar@/_/[ruu]_{h_2} &
  }
\]
\item The pushout of a pair of morphisms $(f_1: M \to N_1, f_2: M \to N_2)$ in $R$-Mod is given by: 
\[ (P= (N_1 \oplus N_2)/ \{ (f_1(m), -f_2(m)) : m \in M\} ,\; g_1: n_1 \mapsto [(n_1, 0)], \; g_2: n_2 \mapsto [(0, n_2)] ). \] 

Moreover,  for every $(Q, h_1, h_2)$ such that $h_1 \circ f_1 = h_2 \circ f_2$, we have that $t: P \to Q$ is given by $t([(n_1, n_2)])= h_1(n_1) + h_2(n_2)$.   

\item (\cite[2.1.13]{prest09}) If $(f_1: M \to N_1, f_2: M \to N_2)$ is a span of pure embeddings in $R$-Mod and $(P, g_1, g_2)$ is the pushout, then $g_1$ and $g_2$ are pure embeddings. 
\end{itemize}

\end{remark}

The next result will be useful to study classes under Hypothesis \ref{hyp2}.  

\begin{lemma}\label{relation}
Let $K \subseteq R\text{-Mod}$ be closed under finite direct sums, pure submodules and isomorphisms, then the following are equivalent:
\begin{enumerate}
\item $K$ is closed under pushouts of pure embeddings in $R$-Mod, i.e., if $M,N_1, N_2 \in K$, $f_1: M \to N_1$ is a pure embedding, $f_2: M \to N_2$ is a pure embeddings and $P$ is the pushout of $(f_1, f_2)$ in $R\text{-}Mod$, then $P \in K$. 
\item $K$ is closed under pure epimorphic images.
\end{enumerate}
\end{lemma}
\begin{proof}
$\Rightarrow$: Assume that the following is a pure-exact sequence:
\[
\xymatrix{
  0 \ar[r] & A \ar[r]^-{i} & B  \ar[r]^-{g} & C  \ar[r] & 0 
}
\]
with $B \in K$. As $A \leq_p B$ and $K$ is closed under pure submodules, it follows that $A \in K$. Then by hypothesis we have $(B \oplus B)/ \{ (a, -a) :  a \in A \} \in K$ because this is the pushout of $(A \hookrightarrow B, A \hookrightarrow B)$.

Define $f: B/A \to (B \oplus B)/ \{ (a, -a) :  a \in A \} $ by $f( b + A) = (b, -b) + \{ (a, -a) :  a \in A \}$.  It is easy to check that $f$ is a pure embeddings. As $K$ is closed under pure submodules, this implies that $B/A \in K$. Hence $C \in K$. 

$\Leftarrow$: Let $A \leq_p B, C$ be a span with $A,B, C \in K$. Observe that $(B \oplus C)/ \{ (a, -a) : a \in A \}$ is the pushout of $(A \hookrightarrow B, A \hookrightarrow C)$. Since $K$ is closed under direct sums $B \oplus C \in K$ and it is straightforward to show that $\pi : B \oplus C \to (B \oplus C)/ \{ (a, -a) : a \in A \}$ is a pure epimorphism. Therefore, $(B \oplus C)/ \{ (a, -a) : a \in A \} \in K$.  \end{proof}

\begin{cor}\label{push}
If $\K$ satisfies Hypothesis \ref{hyp2}, then $K$ is closed under pushouts of pure embeddings in $R$-Mod , i.e., if $M,N_1, N_2 \in K$, $f_1: M \to N_1$ is a pure embedding, $f_2: M \to N_2$ is a pure embeddings and $P$ is the pushout of $(f_1, f_2)$ in $R\text{-}Mod$, then $P \in K$. 
\end{cor}

From the corollary above and closure under direct sums it is clear that if a class satisfies Hypothesis \ref{hyp2}, then it has joint embedding, amalgamation and no maximal models. We record this result for future reference. 

\begin{lemma}\label{stru2}
If $\K$ satisfies Hypothesis \ref{hyp2}, then $\K$ has joint embedding, amalgamation, no maximal models and $\LS(\K)=|R| + \aleph_0$. 
\end{lemma}

Our proof that $\K$ is stable under Hypothesis \ref{hyp2} is longer than that under Hypothesis \ref{hyp1}. This is the case as we do not know if Galois-types are $pp$-syntactic under Hypothesis \ref{hyp2}.\footnote{For torsion groups and $p$-groups this can be done, see \cite[3.4, 4.5]{maztor}.}  The way we proceed is by defining an independence relation in the sense of Subsection 2.2 and showing that it is a weakly stable independence relation with local character.

\begin{defin}\label{indp} Assume $\K$ is an AEC satisfying Hypothesis \ref{hyp2}.
$(f_1, f_2, h_1, h_2) \in \dnf$ if and only if all the arrows of the outer square are pure embeddings and the unique map $t: P \to Q$ is a pure embedding:

 
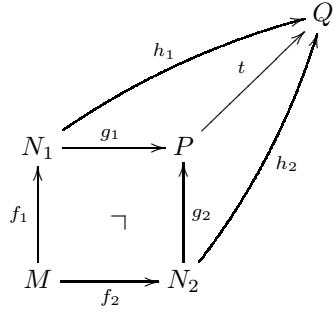
\[
  \xymatrix@=3pc{
    & & Q \\
    N_1 \ar[r]^{g_1}\ar@/^/[rru]^{h_1} & P \ar[ru]^t \pullbackcorner & \\
    M \ar [u]^{f_1} \ar[r]_{f_2} & N_2 \ar[u]_{g_2} \ar@/_/[ruu]_{h_2} &
  }
\]

\end{defin}

\begin{remark}
The definition given above is an instance of \cite[2.2]{lrvcell} where their $\mathcal{K}$ is the category $K$ with morphisms and $\mathcal{M}$ is the class of pure embeddings. Observe that $( \mathcal{K}, \mathcal{M})$ might not be cellular in the sense of \cite{lrvcell} as $\mathcal{K}$ might not be cocomplete.
\end{remark}

Even without the hypothesis that $( \mathcal{K}, \mathcal{M})$ is cellular, one can show as in \cite{lrvcell} that $\dnf$ is a weakly stable independence relation in $\K$ under Hypothesis \ref{hyp2}. The key result is Corollary \ref{push}.

\begin{fact}[{\cite[2.7]{lrvcell}}]\label{wsta}
If $\K$ satisfies Hypothesis \ref{hyp2}, then $\dnf$ is a weakly stable independence relation. 
\end{fact}

\begin{nota}
Given $\dnf$ an independence relation on an AEC, recall that one writes $M_1 \dnf^{N}_M M_2$ if $M \lea M_1, M_2 \lea N$ and $(i_1, i_2, j_1, j_2) \in \dnf$ where $i_1: M \to M_1, i_2:  M \to M_2,j_1: M_1 \to N, j_2: M_2 \to N$ are the inclusion maps. 
\end{nota}
The next result will be essential to describe the stability cardinals.

\begin{theorem}\label{LC}
If $\K$ satisfies Hypothesis \ref{hyp2}, then $\dnf$ has local character. More precisely, if $M_1, M_2 \leq_p N$, then there are $M_1', M_0 \in K$ such that $M_0 \leq_p M_1', M_2 \leq_p N$, $M_1 \leq_p M_1'$, $\| M_0 \| \leq \| M_1 \| + |R| + \aleph_0$ and $M'_1 \dnf^N_{M_0}  M_2$.
\end{theorem}
\begin{proof}
Let $M_1, M_2 \leq_p N$. We build two increasing continuous chains $\{M_{0,i} : i < \omega \}$ and $\{M'_{1,i} : i < \omega \}$ such that:
\begin{enumerate}
\item $M'_{1,0} = M_1$.
\item $M_{0, i} \leq_p M'_{1, i+1}, M_2 \leq_p N$. 
\item $\| M_{0, i} \|, \| M'_{1,i} \| \leq \| M_1 \| + |R| + \aleph_0$.
\item If $\bar{a} \in M'_{1,i}$, $\phi(\bar{x}, \bar{y})$ is a $pp$-formula and there is $\bar{m} \in M_2$ such that $N \vDash \phi(\bar{a}, \bar{m})$, then there is $\bar{l}\in M_{0, i}$ such that $N \vDash \phi(\bar{a}, \bar{l})$. 
\end{enumerate}

\fbox{Construction} \underline{Base:} Let $M'_{1,0} = M_1$. For each $\bar{a} \in M_1$ and $\phi(\bar{x}, \bar{y})$ a $pp$-formula, if there is $\bar{m} \in M_2$ such that $N \vDash \phi(\bar{a}, \bar{m})$ let $\bar{m}_\phi^{\bar{a}}$ be a witness in $M_2$ and $\bar{0}$ otherwise. Let $M_{0,0}$ be the structure obtained by applying Downward L\"{o}wenheim-Skolem to $\bigcup \{ \bar{m}_\phi^{\bar{a}} : \bar{a} \in M_1 \text{ and } \phi \text{ is a }pp\text{-formula} \}$ in $M_2$. It is easy to see that $M_{0,0}$ satisfies what is needed.

\underline{Induction step:} Let $M'_{1, i+1}$ be the structured obtained by applying Downward L\"{o}wenheim-Skolem to $M_{0,i}$ in $N$. Construct $M_{0, i +1}$ as we constructed $M_{0,0}$, but replacing $M'_{1,0}$ by $M'_{1,i}$.

\fbox{Enough} Let $M_0= \bigcup_{ i < \omega } M_{0, i}$ and $M'_1 = \bigcup_{i < \omega } M'_{1,i}$. Observe that $\| M_0 \| \leq  \| M_1 \| + |R| + \aleph_0$ and we show that $M'_1 \dnf^N_{M_0}  M_2$.

Recall that the pushout in $R$-Mod is given by:

 \[
 \xymatrix{\ar @{} [dr] M'_1  \ar[r]  &  (M'_1 \oplus M_2)/ \{ (m, -m) : m \in M_0 \} \pullbackcorner \\
M_0 \ar[u] \ar[r] & M_2 \ar[u]
} 
\]

Moreover, $t:  (M'_1 \oplus M_2)/ \{ (m, -m) : m \in M_0 \} \to N$ is given by $t([(m, n)])= m + n$. So we are left to show that $t$ is a pure embedding.

We begin by proving that $t$ is injective, so assume that $m_1 + n_1 = m_2 + n_2$ with $m_i \in M'_1$ and $n_i \in M_2$ for $i \in \{1, 2\}$. Then $N \vDash x-y = z (m_1, m_2, n_2 -n_1)$, so by condition (4) of the construction there is $m \in M_0$ such that $N \vDash  x-y = z (m_1, m_2, m)$. Hence $[(m_1, n_1)] = [(m_2, n_2)]$ in the pushout. 

We show that $t$ is pure. Let $\phi(y)$ be a $pp$-formula such that $N \vDash  \phi(m+ n)$ with $m \in M_1'$ and $n \in M_2$. So $N\vDash \exists w  ( \phi(w) \wedge w = z + z') (m, n)$. 
Observe that this is a $pp$-formula, $m \in M'_1$ and $n \in M_2$, then by condition (4) of the construction there is $p \in M_0$ such that $N\vDash \exists w  ( \phi(w) \wedge w = z + z') (m, p)$. So $N \vDash \phi(m + p)$. Assume $\phi(y)$ is equal to $\exists \bar{x} \theta(\bar{x}, y)$ for  $\theta(\bar{x}, y)$ quantifier-free formula. Then as $M'_1 \leq_p N$ there is $\bar{m}^\star \in M'_1$ such that \begin{equation}
N \vDash \theta(\bar{m}^\star, m + p).
\end{equation}

As solutions to $pp$-formulas form a subgroup, it is easy to get that $N \vDash \phi(n-p)$. Then as $M_2 \leq_p N$ there is $\bar{n}^\star \in M_2$ such that

\begin{equation}
N \vDash \theta(\bar{n}^\star, n - p).
\end{equation}

So by adding equation (1) and (2) we obtain that:

\begin{equation}
N \vDash \theta(\bar{m}^\star + \bar{n}^\star, m + n).
\end{equation}

Therefore, $t:  (M'_1 \oplus M_2)/ \{ (m, -m) : m \in M_0 \} \to N$ is a pure embedding. \end{proof}

As presented in \cite[8.2]{lrv1}, it is possible to interpret an independence relation $\dnf$ as a relation on Galois-types.  

\begin{defin}

Given $M \leq_p N \in K $, $\bar{a} \in N$ and $B \subseteq N$, we say that $\gtp(\bar{a}/ B; N)$ does not fork over $M$ if and only if there are $M_1, M_2, N' \in K$ such that $\bar{a} \in M_1$, $B \subseteq M_2$, $N \leq_p N'$, $M \leq_p M_1, M_2 \leq_p N'$ and $M_1 \dnf^{N'}_M M_2$
\end{defin}

The next result has some of the properties that the independence relation defined in Definition \ref{indp} has when seen as a relation on Galois-types.

\begin{lemma}\label{lc2} Assume  $\K$ satisfies Hypothesis \ref{hyp2}. Then:
\begin{enumerate}
\item (Uniqueness) If $M \leq_p N$, $p, q \in \gS(N)$, $p, q$ do not fork over $M$ and $p\rest{M} = q\rest{M}$, then $p=q$.
\item (Local character) If $p \in \gS(M)$, then there is $N \leq_p M$ such that $p$ does not fork over $N$ and $\| N \| \leq |R| + \aleph_0$. 
\end{enumerate}
\end{lemma}
\begin{proof}
(1) follows from Fact \ref{wsta} and \cite[8.5]{lrv1}. As for (2), this follows from Theorem \ref{LC}.
\end{proof}

With this we obtain the main result of this section. The proof given is the \emph{standard proof}, but we present the argument for the convenience of the reader.

\begin{theorem}\label{statf} Assume $\K$ satisfies Hypothesis \ref{hyp2}.
 If $\lambda^{|R| + \aleph_0}=\lambda$, then $\K$ is $\lambda$-stable.
\end{theorem}
\begin{proof}
Let $M \in \K_\lambda$ with $\lambda^{|R| + \aleph_0}=\lambda$. Assume for the sake of contradiction that $| \gS(M) | > \lambda$ and let $\{ p_i : i < \lambda^+ \}$ be an enumerations without repetitions of types in $\gS(M)$. 

By Lemma \ref{lc2}, for every $ i < \lambda^+$, there is $N_i \leq_p M$ such that $p_i$ does not fork over $N_i$ and $\| N_i \| = |R| + \aleph_0$. Then by the pigeon hole principle and using that $\lambda^{|R| + \aleph_0}=\lambda$, we may assume that there is an $N \in K$ such that $N_i = N$ for every $i < \lambda^+$. Therefore, by uniqueness, there are $i \neq j < \lambda^+$ such that $p_i = p_j$. This is clearly a contradiction. 
\end{proof} 

The following improves the results of \cite{lrvcell} where it is shown that the class of flat modules with pure embeddings is stable by giving a cardinal arithmetic hypotesis which implies stability. It also extends \cite[4.6]{mazf} where the same result is obtained for those rings such that the pure-injective envelope of every flat module is flat. 

\begin{cor}
If $\lambda^{|R|+\aleph_0}= \lambda$, then $(R\text{-Flat}, \leq_{p})$ is $\lambda$-stable.
\end{cor} 

Moreover, by Theorem \ref{statf} and \cite[3.20]{kuma} we can conclude the existence of universal models.

\begin{cor} Assume $\K$ satisfies Hypothesis \ref{hyp2}.
 If $\lambda^{|R| + \aleph_0}=\lambda$ or $\forall \mu < \lambda( \mu^{|R| + \aleph_0} <
\lambda)$, then $\K$ has a universal model of cardinality $\lambda$.
\end{cor}

\begin{remark}
The above result applied to the class of flat modules extends \cite[4.6]{mazf} which in turn extended \cite[1.2]{sh820}. On the other hand, the above result applied to the class of $\s$-torsion modules extends \cite[4.6]{maztor}. 
\end{remark} 

Another result that follows from having an independence relation is that classes satisfying Hypothesis \ref{hyp2} are tame. 

\begin{lemma}
If $\K$ satisfies Hypothesis \ref{hyp2}, then $\K$ is $(|R| + \aleph_0)$-tame. 
\end{lemma}
\begin{proof}
Follows from Lemma \ref{lc2} and \cite[8.16]{lrv1}.
\end{proof}

Since $K$ has joint embedding, amalgamation and no maximal models, it follows from \cite[\S II.1.16]{shelahaecbook} that $\K$ has a $(\lambda, \alpha)$-limit model if $\lambda^{|R| + \aleph_0}=\lambda$ and $\alpha < \lambda^+$ is a limit ordinal. For classes satisfying Hypothesis \ref{hyp2}, we do not know how limit models look like in general or if there is even a general theory as the one under Hypothesis \ref{hyp1}. For the specific class of flat modules, it was shown that long limit models are cotorsion modules in \cite[3.5]{mazf}. 

Since we were not able to characterize limit models, we are not able to characterize superstability for classes satisfying Hypothesis \ref{hyp2}. Again, for the class of flat modules this was done in \cite{mazf}. There it was shown that the class of flat left $R$-modules is superstable if and only if $R$ is left perfect.

We are not sure if it is possible to obtain a result as Theorem \ref{main} for classes satisfying Hypothesis \ref{hyp2}, but we think that characterizing superstability in the class of $\s$-torsion $R$-modules will have interesting algebraic consequences.

\subsection{Classes satisfying Hypotheses \ref{hyp1} and \ref{hyp2}}  We briefly study those classes that satisfy Hypotheses \ref{hyp1} and \ref{hyp2}. Recall that definable classes and Example \ref{ex1}.(5) are examples of classes satisfying both hypotheses.

\begin{lemma}\label{easy}
If $\K$ satisfies Hypotheses \ref{hyp1} and \ref{hyp2}, then $\dnf$ has the $(<\aleph_0)$-witness property. Moreover, $\dnf$ is a stable independence relation.
\end{lemma}
\begin{proof}
By Corollary \ref{fullt} we have $\K$ is fully $(< \aleph_0)$-tame. Then it follows from \cite[8.8, 8.9]{lrv1} that $\dnf$ has the $(<\aleph_0)$-witness property. The moreover part follows from Fact \ref{wsta} and Theorem \ref{LC}.
\end{proof}

A natural question to ask is if the above results follows from Hypothesis \ref{hyp2}. 

\begin{question}\label{qq}
If $\K$ satisfy Hypothesis \ref{hyp2}, is $\dnf$ a stable independence relation?
\end{question}

\begin{remark}
In the case of $p$-groups and torsion groups this is the case by \cite[3.4, 4.5]{maztor}, Lemma \ref{pp=gtp} and doing a similar argument as that of Lemma \ref{easy}.\footnote{After the original submission of this paper, the same result was obtained in \cite[\S 3]{lrv21} using completely different methods. There it is shown that many classes of modules have a stable independence relation. Nevertheless Question \ref{qq} is still open.}
\end{remark} 

The next assertion follows from the previous lemma and \cite[3.1]{lrvcell}. For the notions not defined in this paper, the reader can consult \cite{lrvcell}.

\begin{cor}
Pure embeddings are cofibrantly generated in the class of $R$-modules, i.e., they are generated from a set of morphisms by pushouts, transfinite composition and retracts.
\end{cor} 
\begin{proof}
Observe that the class of left $R$-modules with pure embeddings satisfies Hypotheses \ref{hyp1} and \ref{hyp2}, then by Lemma \ref{easy} $\dnf$ is a stable independence relation. Since $R$-Mod with pure embeddings is an accessible cellular category which is retract-closed, coherent and $\aleph_0$-continuous. Therefore, pure embeddings are cofibrantly generated by \cite[3.1]{lrvcell}. 
\end{proof}

\begin{remark}
The main result of \cite{lprv} is that the above result holds in locally finitely accessible additive categories. Their proof is very different from our proof as they use categorical methods. 
\end{remark}

\section{Classes that admit intersections}

In this section we study classes that admit intersections and their subclasses. We use the ideas of this section to provide a partial solution to Question \ref{mainq} for AECs of torsion-free abelian groups. Moreover, we give a condition that implies a positive solution to Question \ref{mainq}.

\begin{defin}
Let $\K= (K, \leq_p)$ and $\K^\star=(K^\star, \leq_p)$ be a pair of AECs with $K, K^\star \subseteq R\text{-Mod}$ for a fixed ring $R$. We say $\K^\star$ is \emph{closed below} $\K$ if the following hold:
\begin{enumerate}
\item $K^\star \subseteq K$.
\item $K$ and $K^\star$ are closed under pure submodules.
\item $\K$ admits intersections, i.e., for every $N \in K$ and $A \subseteq |N|$ we have that $cl^{N}_{\K}(A)=\bigcap\{M \leq_p N : A \subseteq |M|\} \in K$ and  $cl^{N}_{\K}(A) \leq_p N$.\footnote{Classes admitting intersections were introduced in \cite[1.2]{BSh} and studied in detail in \cite[$\S$2]{vaseyu}.}
\end{enumerate}
\end{defin}

\begin{example}\label{ex3}
The following classes are all closed below the class of torsion-free groups with pure embeddings:
\begin{enumerate}
\item $(TF, \leq_p)$ where $TF$ is the class of torsion-free groups. A group $G$ is torsion-free if every element has infinite order. 
\item $(RTF, \leq_p)$ where \text{RTF} is the class of reduced torsion-free abelian groups. A group $G$ is reduced if it does not have non-trivial divisible subgroups. 
\item $(\aleph_1\text{-free}, \leq_p)$ where $\aleph_1\text{-free}$ is the class of $\aleph_1\text{-free}$ groups. A group $G$ is $\aleph_1\text{-free}$ if every countable subgroups is free. 
\item $(B_0, \leq_p)$ where $B_0$ is the class of finitely Butler groups.   A group $G$ is a finitely Butler group if $G$ is torsion-free and every pure subgroup of finite rank is a pure subgroup of a finite rank completely decomposable group (see \cite[\S 14.4]{fucb} for more details). 
\item $(TF\text{-l-cyc}, \leq_pp)$ where TF-l-cyc is the class of torsion-free locally cyclic groups.  A group $G$ is locally cyclic if  every finitely generated subgroup is cyclic. 
\end{enumerate} 
\end{example} 

\begin{remark}
It is worth pointing out that the second, third and fifth example are not first-order axiomatizable while the fourth one is probably not first-order axiomatizable. 
\end{remark}

\begin{remark}
The class of $\aleph_1\text{-free}$ groups is closed below the class of torsion-free groups, but does not satisfy Hypothesis \ref{hyp1} or Hypothesis \ref{hyp2}. This is the case as it does not have the amalgamation property. We showed that if a class satisfied either of the hypotheses then it had the amalgamation property (Lemma \ref{stru1} and Lemma \ref{stru2}). 

$(R\text{-Mod}, \leq_p)$ satisfies Hypothesis \ref{hyp1} and Hypothesis \ref{hyp2}, but it is not closed below any class of modules for most rings. For example, if $R = \mathbb{Z}$, this is the case as the class of abelian groups with pure embeddings does not admit intersections. 

 Therefore, there are classes studied in this section that do not satisfy Hypotheses \ref{hyp1} or \ref{hyp2} and there are classes satisfying those hypotheses that can not be handled with the methods of this section. 

\end{remark} 

\subsection{Stability} The proof of the next result is straightforward so we omit it. 

\begin{prop} 
If $\K^\star$ is closed below $\K$, then $\K^\star$ admits intersections. Moreover, for every $N \in \K^\star$ and $A \subseteq N$ we have that $cl^{N}_{\K}(A)=cl^{N}_{\K^\star}(A)$.
\end{prop}

With it we can show that there is a close relation between Galois-types in $\K$ and $\K^\star$.

\begin{lemma}
Assume $\K^\star$ is closed below $\K$. Let $A \subseteq N_1, N_2 \in K^\star$, $\bar{a} \in N_1^{< \infty}$ and $\bar{b} \in N_2^{<\infty}$, then:

\[ \gtp_{\K}(\bar{a}/ A; N_1) = \gtp_{\K}(\bar{b}/A ; N_2) \text{ if and only if } \gtp_{\K^\star}(\bar{a}/ A; N_1) = \gtp_{\K^\star}(\bar{b}/A ; N_2) \]
\end{lemma}
\begin{proof}
The backward direction is obvious so we prove the forward direction. Since $\K$ admits intersection, by \cite[2.18]{vaseyu}, there is $f: cl^{N_1}_{\K}(\bar{a} \cup A) \cong_M cl^{N_2}_{\K}(\bar{b} \cup A)$ with $f(\bar{a})=\bar{b}$.  Then using the proposition above we have that $cl^{N_1}_{\K}(\bar{a} \cup A) = cl^{N_1}_{\K^\star}(\bar{a} \cup A)$  and $cl^{N_2}_{\K}(\bar{b} \cup A)= cl^{N_2}_{\K^\star}(\bar{b} \cup A)$. So the result follows from the fact that $\K^\star$ admits intersections and \cite[2.18]{vaseyu}.
\end{proof}

From that characterization we obtain the following.

\begin{cor}
Assume $\K^\star$ is closed below $\K$. 
\begin{enumerate}
\item Let $\lambda \geq \LS(\K^\star)$. If $\K$ is $\lambda$-stable, then $\K^\star$ is $\lambda$-stable.
\item  Let $\lambda$ be an infinite cardinal. If $\K$ is $(<\lambda)$-tame, then $\K^\star$ is $(<\lambda)$-tame.
\item If Galois-types in $\K$ are $pp$-syntactic, then Galois-types in $\K^\star$ are $pp$-syntactic.
\end{enumerate} 
\end{cor}

Using the above result together with Theorem \ref{stat1} we are able to answer Question \ref{mainq} in the case of AECs of torsion-free abelian groups closed under pure submodules and with arbitrary large models.

\begin{lemma}
If $\K=(K, \leq_p)$ is an AEC closed under pure submodules and with arbitrary large models such that $K \subseteq TF$, then $\K$ is $\lambda$-stable for every infinite cardinal $\lambda$ such that $\lambda^{\aleph_0}=\lambda$.
\end{lemma}

\begin{remark}
The above result applies in particular to reduced torsion-free groups, $\aleph_1$-free groups and finitely Butler groups. The result for reduced torsion-free groups is in \cite[1.2]{sh820}, for $\aleph_1$-free groups is in \cite[2.9]{maztor}, and for finitely Butler groups is in \cite[5.9]{maz}.
\end{remark}

We see the next result as a weak approximation to Question \ref{mainq}. Recall that a ring $R$ is Von Neumann regular if and only if for every $r \in R$ there is an $s \in R$ such that $r= rsr$ if and only if every left $R$-modules is absolutely pure (see for example \cite[2.3.22]{prest09}).

\begin{lemma}\label{von}
Assume $R$ is a Von Neumann regular ring.
If $K$ is closed under submodules and has arbitrarily large models, then $\K =(K, \leq_p)$ is $\lambda$-stable for every infinite cardinal $\lambda$ such that $\lambda^{ |R| + \aleph_0} = \lambda$.
\end{lemma}
\begin{proof}
We show that $\K$ is closed below $(R\text{-Mod}, \leq_p)$. Observe that the only things that need to be shown are that $(R\text{-Mod}, \leq_p)$ admits intersections and that $K$ is closed under pure submodules. This is the case as every module is absolutely pure by the hypothesis on the ring.
\end{proof}


\end{document}